\documentclass[a4paper, 11pt, twoside, leqno, openright]{amsart}

\usepackage[top=1.5in, bottom=1.2in, left=1.2in, right=1.2in]{geometry}

\usepackage{amssymb}
\usepackage{amsthm}
\usepackage{amscd}
\usepackage{mathrsfs}
\usepackage{amsopn}
\usepackage{manfnt}

\usepackage{latexsym}
\usepackage{amsmath}
\usepackage{epic,eepic}
\usepackage{pifont}
\usepackage{lmodern}

\usepackage[pdftex]{graphicx}
\usepackage{tikz}
\usepackage{tikz-cd}
\usepackage{pgfplots}

\usetikzlibrary{intersections}
\usetikzlibrary{positioning}
\usetikzlibrary{arrows}
\usetikzlibrary{patterns}
\usetikzlibrary{decorations.pathreplacing}

\usepackage{url}
\usepackage[all]{xy}
\usepackage{paralist}
\usepackage{hyperref}

\newcommand{\longhookrightarrow}
{\DOTSB\lhook\joinrel\longrightarrow}

\def\bN{{\mathbb N}}
\def\bZ{{\mathbb Z}}
\def\bQ{{\mathbb Q}}

\def\bC{{\mathbb C}}

\def\an{{\mathrm{an}}}

\def\top{{\mathrm{top}}}

\def\exp{{\mathrm{exp}}}

\def\Aut{{\mbox{Aut}}}

\def\Gal{{\mbox{Gal}}}

\def\Spec{{\mbox{Spec~}}}

\def\p{{\mathfrak{p}}}

\newtheorem{dfn}{Definition}[section]

\newtheorem{thm}[dfn]{Theorem}

\newtheorem{prop}[dfn]{Proposition}

\newtheorem{cor}[dfn]{Corollary}
\newtheorem{rem}[dfn]{Remark}

\theoremstyle{definition}
\newtheorem{theorem}{Theorem}
\newtheorem*{theorem*}{Theorem}
\newtheorem{definition}[theorem]{Definition}
\newtheorem*{definition*}{Definition}
\newtheorem{remark}[theorem]{Remark}
\newtheorem*{Remark*}{Remark}

\usepackage{fancyhdr} 
\makeatletter

\@addtoreset{equation}{subsection}
\makeatother

\begin{document}

\title{\vspace{-2mm}On $\ell$-adic 
Galois polylogarithms
\\~and triple $\ell$-th power residue symbols}
\author[D.~Shiraishi]{Densuke Shiraishi}
\address{Department of Mathematics,
Graduate School of Science,
Osaka University, Toyonaka, Osaka 560-0043, Japan}
\email{densuke.shiraishi@gmail.com,
u848765h@ecs.osaka-u.ac.jp}
\date{}

\maketitle

\vspace{-6mm}

\begin{abstract}
The $\ell$-adic Galois polylogarithm 
is an arithmetic function 
on an absolute Galois group with values 
in $\ell$-adic numbers,
which arises from Galois actions 
on $\ell$-adic $\rm{\acute{e}}$tale paths 
on ${\mathbb P}^1 \backslash \{0,1,\infty\}$.
In the present paper,
we discuss a relationship between 
$\ell$-adic Galois polylogarithms 
and triple $\ell$-th power residue symbols 
in some special cases
studied by a work of Hirano-Morishita [HM].
We show that a functional equation 
of $\ell$-adic Galois polylogarithms 
by Nakamura-Wojtkowiak [NW2] implies 
a reciprocity law of triple $\ell$-th 
power residue symbols.
\end{abstract}

\section*{Introduction}

Let $K$ be a number field,
$\overline{K}$ its algebraic closure in the complex number field $\bC$.
For a prime number $\ell$,
let $\zeta_{\ell}
:=\exp(\frac{2\pi\sqrt{-1}}{\ell})$
a primitive $\ell$-th root of unity in $\overline{K}$. 

Choose a $K$-rational point $z$ 
of ${\mathbb P}_{{K}}^1\backslash 
\{0,1,\infty\}$.
For any prime number $\ell$,
the absolute Galois group 
$G_K:=\Gal(\overline{K}/K)$ 
acts on the $\ell$-adic 
$\rm{\acute{e}}$tale path space
$\pi^{\ell}_1
({\mathbb P}_{\overline{K}}^1\backslash 
\{0,1,\infty\}; \overrightarrow{01},\bar{z})$
where $\overrightarrow{01}$ is 
the standard $K$-rational tangential base point.
In [Wo],
for a fixed $\ell$-adic $\rm{\acute{e}}$tale path $\gamma 
\in \pi^{\ell}_1
({\mathbb P}_{\overline{K}}^1\backslash 
\{0,1,\infty\}; \overrightarrow{01},\bar{z})$,
Z. Wojtkowiak introduced 
an arithmetic function
\[\ell i_{n}^{(\ell)}(z,\gamma):
G_K \to \bQ_{\ell}\]
(for $n=2,3,4,\dots$)
valued in the $\ell$-adic number field $\bQ_\ell$,
called the $n$-th 
{\it $\ell$-adic Galois polylogarithm},
defined as a certain coefficient
in the $\ell$-adic Magnus expansion
of the loop
\[{\mathfrak f}^{(\ell)}_{\gamma}(\sigma):=
\gamma \cdot \sigma(\gamma)^{-1}
\in \pi^{\ell}_1
({\mathbb P}_{\overline{K}}^1\backslash 
\{0,1,\infty\}, \overrightarrow{01})
~\qquad~(\sigma \in G_K).\]

On the other hand,
following the analogy
between knots and primes,
M. Morishita introduced
the mod $\ell$ Milnor invariant 
$\mu_\ell(123) \in \bZ/\ell \bZ$ 
for certain prime ideals $\p_1,\p_2,\p_3$ 
of $\bQ(\zeta_\ell)$ for $\ell=2,
3$,
as an arithmetic analog 
of the Milnor invariant of links
([Mo], [AMM]).
As a result,
the 
{\it triple $\ell$-th power residue symbol}
is defined by
\[[\p_1,\p_2,\p_3]_{\ell}
:=\zeta_\ell^{\mu_\ell(123)},\]
which controls
the decomposition law of $\p_3$
in a certain nilpotent extension
$R_{\p_1,\p_2}^{(\ell)}/\bQ(\zeta_{\ell})$.

In the present paper,
we relate $[\p_1,\p_2,\p_3]_{\ell}$ to 
$\ell i_{n}^{(\ell)}(z,\gamma)$ for $\ell=2,
3$ as follows:\\
{\bf Main formula}~{\rm (Naive form)}{\bf .}
~\hspace{1mm}~{\rm For} 
$\ell \in \{2,
3\}$,
we have
\begin{equation}\label{formula}
[\p_1,\p_2,\p_3]_{\ell} 
= \pm \zeta_{\ell}
^{-\ell i_{2}^{(\ell)}(z,\gamma)(\sigma)},
\end{equation}
\noindent
where $K,
z,
\gamma,
\sigma$ are suitably chosen 
to satisfy certain 
conditions depending 
on the triple of primes $\{\p_1,\p_2,\p_3\}$.
(See Theorem~\ref{tl} for more details.)

Moreover,
as a consequence of (\ref{formula}),
we derive
a reciprocity law 
of the triple symbol $[\p_1,\p_2,\p_3]_{\ell}$
due to R$\rm{\acute{e}}$dei 
[{R${\rm \acute{e}}$}],
Amano-Mizusawa-Morishita [AMM]
in the form
\begin{equation}\label{functional}
[\p_1,\p_2,\p_3]_{\ell} \cdot 
[\p_2,\p_1,\p_3]_{\ell} 
=1~\qquad~(\ell = 2,
3)
\end{equation}
from a functional equation between
$\ell i_{2}^{(\ell)}(z,\gamma)$
and
$\ell i_{2}^{(\ell)}(1-z,\gamma')$
due to Nakamura-Wojtkowiak [NW2].
(See Corollary~\ref{reci} for details.)
Thus,
by using a functional equation 
of $\ell$-adic Galois polylogarithms,
we have another proof of 
a reciprocity law of triple $\ell$-th 
power residue symbols.
This fact is an indication that 
the Galois action mentioned at 
the beginning of this introduction 
has abundant arithmetic information.

\section*{Acknowledgements}

This work is inspired by 
the work of Hirano-Morishita [HM]
and is based on my Master's thesis [Sh].
I would like to express 
my deepest gratitude 
to my supervisor,
Professor Hiroaki Nakamura for
his helpful advice
and warm encouragement.
I am deeply grateful to
Professor Masanori Morishita 
for inviting me to the Workshop
``Low dimensional topology and number theory XI''
and for providing useful comments on 
improving the expression of main theorems 
in the present paper during the workshop.
I also wish to express my thanks to Professor 
Yasushi Mizusawa who told me how to 
compute the triple symbol $[\p_1,\p_2,\p_3]_{3}$ 
by using  PARI/GP.


\section{Preliminaries}


\subsection{\boldmath $\ell$-adic 
Galois polylogarithms}\label{lgp}

In this section,
we recall the definition 
and some properties 
of $\ell$-adic Galois polylogarithms.
Fix any prime number $\ell$.
Let $K$ be a sub-field of $\bC$,
$\overline{K}$ 
a fixed algebraic closure of $K$,
and $G_K:=\Gal( \overline{K}/K)$ 
the absolute Galois group of 
$K$ with respect to $\overline{K}$.
Fix an embedding $\overline{K} \hookrightarrow \bC$.
For any positive integer $m$,
let $\zeta_{\ell^m}
:=\exp(\frac{2\pi\sqrt{-1}}{\ell^m})$
a fixed primitive $\ell^m$-th root of unity in $\overline{K}$.
Let
\[X:={\mathbb P}_K^1
\backslash \{0,1,\infty\}\]
be a projective line minus 3 points 
over $K$,
$X_{\overline{K}}
:=X\times_K \overline{K}$ 
the base change 
of $X \to \Spec K$ 
via $K \hookrightarrow \overline{K}$,
and
$X^{\an}={\mathbb P}^1
(\bC)\backslash \{0,1,\infty\}$ 
the complex analytic space
associated to the base change
of $X_{\overline{K}}$ via $\overline{K} \hookrightarrow \bC$.

First,
we choose
a $K$-rational point \[z \in X(K).\]
As appropriate,
we regard $z$ as a point on $X^{\an}$ by the fixed embeddings
$K \hookrightarrow \overline{K}$ and $\overline{K} \hookrightarrow \bC$.
We denote by 
$\pi_1
^{\top}(X^{\an}; \overrightarrow{01}, z)$
the set of homotopy classes
of 
piece-wise smooth topological paths on $X^{\an}$
from 
the unit tangent vector $\overrightarrow{01}$
to $z$;
by $\pi_1^\top(X^{\an}, \overrightarrow{01})$
the topological fundamental group 
of $X^{\an}$ with the base point $\overrightarrow{01}$.
The group
$\pi_1^\top(X^{\an}, \overrightarrow{01})$
is a free group of rank $2$
generated by
the homotopy classes of $\{l_0,
l_1\}$ in the following Figure:
\[\pi_1
^{\rm top}(X^{\rm an}, \overrightarrow{01})={\left<l_0,
l_1\right>}.\]
Fix a homotopy class
\[\gamma \in \pi_1
^{\top}(X^{\an}; \overrightarrow{01}, z).\]

Let $\bar{z}: \Spec \overline{K} 
\to X_{\overline{K}}$ 
be the base change 
of $z: \Spec K \to X$ 
via $K \hookrightarrow \overline{K}$.
We denote
by $\pi_1^{\ell}
(X_{\overline{K}}; \overrightarrow{01}, \bar{z})$
the profinite set of $\ell$-adic $\rm{\acute{e}}$tale paths
on $X_{\overline{K}}$
from the standard $K$-rational 
tangential base point $\overrightarrow{01}$
to $\bar{z}$;
by $\pi_1
^\ell
(X_{\overline{K}}, \ast)$ the pro-$\ell$ $\rm{\acute{e}}$tale
fundamental group of $X_{\overline{K}}$ with 
base point $\ast \in \{\overrightarrow{01}, \bar{z}\}$.
By using the comparison maps induced 
by  the fixed embedding 
$\overline{K} \hookrightarrow \bC$,
we regard homotopy classes $l_0,
l_1 \in \pi_1
^{\top}(X^{\an}, \overrightarrow{01})$,
$\gamma \in \pi_1
^{\top}(X^{\an}; \overrightarrow{01}, z)$ 
as $\ell$-adic 
$\rm{\acute{e}}$tale paths
\[l_0,
l_1 \in \pi_1^\ell
(X_{\overline{K}}, \overrightarrow{01}),
\gamma \in \pi_1^{\ell}
(X_{\overline{K}}; \overrightarrow{01}, \bar{z}).\]
Then
$\pi_1^\ell
(X_{\overline{K}}, \overrightarrow{01})$
is a free pro-$\ell$ group 
topologically generated 
by $\{l_0,
l_1\}$:
\[\pi_1^{\ell}(X_{\overline{K}}, \overrightarrow{01})
=\overline{\left<l_0,
l_1\right>}.\]

\vspace{0.5cm}
\begin{center}
\begin{tikzpicture}
\draw (2,0) -- (6,0);
\draw (4.6,0) -- (4.5,0.1);
\draw (4.6,0) -- (4.5,-0.1);
\draw (0.7,0) -- (0.6,0.1);
\draw (0.7,0) -- (0.8,0.1);
\draw (7.3,0) -- (7.4,-0.1);
\draw (7.3,0) -- (7.2,-0.1);
\draw (3.67,1.1) -- (3.48,1.1);
\draw (3.67,1.1) -- (3.67,0.9);
\draw (2.1,0) to [out=0,in=270] (2.8,0.7);
\draw (2.8,0.7) to [out=90,in=360] (2,1.4);
\draw (2,1.4) to [out=180,in=90] (0.7,0);
\draw (0.7,0) to [out=270,in=180] (2,-1.4);
\draw (2,-1.4) to [out=0,in=270] (2.8,-0.7);
\draw (2.8,-0.7) to [out=90,in=360] (2.1,0);
\draw (2,0) to [out=0,in=210] (4,0.4);
\draw (4,0.4) to [out=30,in=180] (6,1.4);
\draw (6,1.4) to [out=0,in=90] (7.3,0);
\draw (7.3,0) to [out=270,in=0] (6,-1.4);
\draw (4,-0.4) to [out=330,in=180] (6,-1.4);
\draw (4,-0.4) to [out=150,in=0] (2,0);
\draw (2.1,0) to [out=0,in=180] (5,1.9);
\node at (0.3,0) {$l_0$};
\node at (5,1.9) {${\bullet}$};
\node at (5.4,1.9) {$z$};
\node at (3.6,1.6) {$\gamma$};
\node at (1.8,-0.4) {$0$};
\node at (4.7,-0.35) {$\delta$};
\node at (6.2,-0.4) {$1$};
\node at (7.7,0) {$l_1$};
\node at (2,0) {${\bullet}$};
\node at (6,0) {${\bullet}$};
\node at (6,0) {${\bullet}$};
\node at (4,-2.3) {\bf{Figure of} $\boldsymbol{X^\an}$};
\end{tikzpicture}
\end{center}
\vspace{0.5cm}

Next,
we focus on the Galois action
\[G_K \to \Aut(\pi_1^{\ell}
(X_{\overline{K}}; \overrightarrow{01}, \bar{z}))\]
defined by
$\sigma(p):=s_{\overrightarrow{01}}(\sigma) \cdot p \cdot s_{\bar{z}}(\sigma)^{-1}$
for $\sigma \in G_K$ and $p \in \pi_1^{\ell}
(X_{\overline{K}}; \overrightarrow{01}, \bar{z})$,
where $s_{\ast}: G_K \to \pi_1^{\ell}
(X_{\overline{K}}, \ast)$ is a canonical homomorphism 
induced by a geometric point $\ast \in \{\overrightarrow{01}, \bar{z}\}$ 
on $X_{\overline{K}}$
and paths are composed from left to right.
Consider 
the continuous $1$-cocycle
\[{\mathfrak f}^{(\ell)}_{\gamma}: 
G_K \to \pi_1
^\ell
(X_{\overline{K}}, \overrightarrow{01})\]
defined by ${\mathfrak f}
^{(\ell)}_{\gamma}(\sigma):=
\gamma \cdot \sigma(\gamma)^{-1} ~\in 
\pi_1
^\ell(X_{\overline{K}}, \overrightarrow{01})$.
To understand clearly the behavior 
of ${\mathfrak f}^{(\ell)}_{\gamma}$,
we use 
the $\ell$-adic Magnus embedding
\[E:\pi_1^{\ell}
(X_{\overline{K}}, \overrightarrow{01}) 
\longhookrightarrow 
\mathbb{Q}_{\ell} 
\langle \langle e_0,e_1 \rangle \rangle \]
defined by $E(l_0)={\rm exp}(e_0),
E(l_1)={\rm exp}(e_1)$
where 
$\mathbb{Q}_{\ell} 
\langle \langle e_0,e_1 \rangle \rangle$
is the $\mathbb{Q}_{\ell}$-algebra 
of formal power series 
over $\bQ_{\ell}$ in 
two non-commuting variables $e_0$
and $e_1$.
The constant term of $E({\mathfrak f}
^{(\ell)}_{\gamma}(\sigma)) 
\in \mathbb{Q}_{\ell} 
\langle \langle e_0,e_1 \rangle \rangle$
is equal to $1$ for any $\sigma \in G_K$,
so we can consider 
the Lie formal power series ${\rm log}
(E({\mathfrak f}_{\gamma}
^{(\ell)}(\sigma)))^{-1} \in {\rm Lie} 
\langle \langle e_0,e_1 \rangle \rangle 
\subset \mathbb{Q}_{\ell} 
\langle \langle e_0,e_1 \rangle \rangle$
where ${\rm Lie} 
\langle \langle e_0,e_1 \rangle \rangle$ is 
the complete free Lie algebra generated by $e_0$
and $e_1$.

Now
we shall introduce a certain function 
on $G_K$ which quantifies 
the loop ${\mathfrak f}
^{(\ell)}_{\gamma}(\sigma) \in \pi_1
^\ell(X_{\overline{K}}, \overrightarrow{01})
~(\sigma \in G_K)$ 
as a ``polylogarithm'' with values 
in $\ell$-adic numbers.
To intoroduce it,
we need the following preparations.
Let $\delta \in \pi_1
^{\top}(X^\an; \overrightarrow{01},
\overrightarrow{10})$
be a homotopy class of the canonical path on $X^\an$
as in the above Figure,
and
\[\gamma':= \delta \cdot \varphi(\gamma) 
\in \pi_1^{\top}
(X^\an; \overrightarrow{01}, 1-z),\]
where $\varphi \in {\rm Aut}(X^\an)$ 
given by $\varphi(*)=1-*$. 
We will choose $z^{1/n},
(1-z)^{1/n},
(1-\zeta^a_n z^{1/n})^{1/m}$
($
n,
m \in \bN,
a \in \bZ$)
as the specific $n$-th power roots 
determined by 
$\gamma 
\in \pi_1^{\top}
(X^{\an}; \overrightarrow{01}, z)$
(See [NW1] for details).
Let 
\[\rho_{z, \gamma}
: G_K \to \bZ_\ell~\qquad~
({\rm resp.~} \rho_{1-z, \gamma'}
: G_K \to \bZ_\ell)\]
be the Kummer 1-cocycle 
along
$\gamma$ (resp. $\gamma'$) defined by
$\sigma(z^{1/\ell^n})=
\zeta_{\ell^n}
^{\rho_{z, \gamma}(\sigma)}z^{1/\ell^n}$
(resp. $\sigma((1-z)^{1/\ell^n})=
\zeta_{\ell^n}
^{\rho_{1-z, \gamma'}(\sigma)}(1-z)^{1/\ell^n}$)
for $\sigma \in G_K$.
Denote by $\chi
: G_K \to \bZ_{\ell}^{\times}$
the $\ell$-adic cyclotomic character 
defined by
$\sigma(\zeta_{\ell^n})
=\zeta_{\ell^n}^{\chi(\sigma)}$
for $\sigma \in G_K$.

\begin{definition}[$\ell$-adic 
Galois polylogarithm function{\rm ;}
${\rm ~[NW1], [Wo; \S11]}$]\label{func}
We define a function $\ell i_{n}
^{(\ell)}(z,\gamma):
G_K \to \bQ_{\ell}~(n \geq 2)$
as a coefficient of ${\rm ad}(e_0)
^{n-1}(e_1)$ 
in the following Lie expression of 
${\rm log}(E({\mathfrak f}_{\gamma}
^{(\ell)}(\sigma)))^{-1}$
~for any $\sigma \in G_K$:
\[{\rm log}(E({\mathfrak f}_{\gamma}
^{(\ell)}(\sigma)))^{-1} \equiv 
\rho_{z, \gamma}(\sigma)e_0
+\rho_{1-z, \gamma'}(\sigma)e_1\hspace{2.5cm}\]
\[\hspace{5cm}+\sum_{n=2}
^{\infty}\ell i_{n}^{(\ell)}
(z,\gamma)(\sigma){\rm ad}(e_0)
^{n-1}(e_1)~\qquad~{\rm mod}~I_{e_1},\]
where
$I_{e_1}$ denotes the ideal generated by 
Lie monomials involving $e_1$ at least twice.
This function \[\ell i_{n}^{(\ell)}(z,\gamma)
:G_K \to \bQ_{\ell}~\qquad~(n \geq 2)\] 
is called the $n$-th {\it $\ell$-adic Galois 
polylogarithm function} 
associated to $\gamma \in \pi_1^{\top}
(X^\an; \overrightarrow{01}, z)$.
We shall also define
\[\ell i_{0}^{(\ell)}(z,\gamma)
:=\rho_{z, \gamma},~
\ell i_{1}^{(\ell)}(z,\gamma)
:=\rho_{1-z, \gamma'}.\]
\end{definition}

Here
we shall introduce a certain character 
on $G_K$ 
which generalizes 
the so-called Soul$\rm{\acute{e}}$ character.

\begin{definition}[$\ell$-adic Galois polylogarithmic 
character{\rm ;}${\rm ~[NW1]}$]\label{char}
For any integer $m \geq 1$,
we define $\tilde{\chi}_{m}
^{z, \gamma}:G_K \to \bZ_{\ell}$ 
by the following Kummer properties:
\[\zeta_{\ell^n}^{\tilde{\chi}_{m}
^{z, \gamma}(\sigma)}=
\sigma \left(\prod_{i=0}^{\ell^n -1}
(1-\zeta_{\ell^n}^{\chi(\sigma)
^{-1}i}z^{1/\ell^n})
^{\frac{i^{m-1}}{\ell^n}}\right) 
\bigg/ \prod_{i=0}^{\ell^n -1}
(1-\zeta_{\ell^n}
^{i+\rho_{z, \gamma}(\sigma)}z^{1/\ell^n})
^{\frac{i^{m-1}}{\ell^n}}~~~(n \geq 1). \]
This function 
\[\tilde{\chi}_{m}^{z, \gamma}:
G_K \to \bZ_{\ell}~\qquad~(m \geq 1)\] 
valued in the ring $\bZ_\ell$ of $\ell$-adic integers,
is called the $m$-th {\it $\ell$-adic Galois 
polylogarithmic character}\hspace{0.1cm}
associated to $\gamma \in \pi_1^{\top}
(X^\an; \overrightarrow{01}, z)$.
\end{definition}

In fact,
$\ell$-adic Galois 
polylogarithmic characters describe
values 
of the $\ell$-adic 
Galois polylogarithm function.

\begin{thm}[Explicit formula{\rm ;}
${\rm~[NW1; Corollary]}$]\label{explicit}
For each $\sigma \in G_K$,
the quantity $\ell i_{n}^{(\ell)}(z,\gamma)(\sigma)$ 
is explicitly described 
by $\ell$-adic 
Galois polylogarithmic characters 
as follows:
\[\displaystyle \ell i
^{(\ell)}_{n}(z,\gamma)(\sigma)
=(-1)^{n+1}\sum_{k=0}^{n-1}\dfrac{B_k}{k!}
(-\rho_{z, \gamma}(\sigma))
^{k}\dfrac{\tilde{\chi}_{n-k}
^{z, \gamma}(\sigma)}{(n-k-1)!}
~\qquad~(n \geq 1),\]
where
$B_k$ denotes the $k$-th Bernoulli number.
\end{thm}

One reason for 
the name  ``$\ell$-adic Galois polylogarithm'' 
is that $\ell$-adic Galois polylogarithm 
functions/polylogarithmic characters 
satisfy some typical functional equations 
analogous to functional equations 
of the classical polylogarithm [NW2;~Chapter~6].
The following functional equation is 
one example of them.

\begin{thm}[a functional equation{\rm ;}
${\rm ~[NW2;~Chapter~6,~(6.14)]}$]\label{func}
The 2nd 
$\ell$-adic Galois polylogarithm function
holds the following functional equation.
For any $\sigma \in G_K$,
\[
\ell i_{2}^{(\ell)}(z,\gamma)(\sigma)
+\ell i_{2}^{(\ell)}
(1-z,\gamma')(\sigma)=\ell i_{2}^{(\ell)}
(\overrightarrow{10},\delta)(\sigma).
\]
By Theorem \ref{explicit},
this equation is equivalent to 
the following functional equation 
of the 2nd 
$\ell$-adic Galois polylogarithmic character.
For any $\sigma \in G_K$,
\[\tilde{\chi}_2^{z, \gamma}(\sigma)
+\tilde{\chi}_2^{1-z, \gamma'}(\sigma)
+\rho_{z, \gamma}(\sigma)
\rho_{1-z, \gamma'}(\sigma)
=\dfrac{1}{24}(\chi(\sigma)^2-1).\]
\end{thm}

\begin{remark}
The latter functional equation 
in Theorem \ref{func} is 
an $\ell$-adic Galois analog 
of the functional equation 
\[{Li}_2(z)+{Li}_2(1-z)
+{\rm log}(z){\rm log}(1-z)
=\dfrac{\pi^2}{6},\]
where ${Li}_2(z)$ denotes 
the classical dilogarithm function.
\end{remark}

\subsection{\boldmath Triple $\ell$-th 
power residue symbols for $\ell=2,
3$}\label{triplesymbol}

The triple $\ell$-th 
power residue symbol 
is defined at present for $\ell=2,
3$ in [Mo],
[AMM].
In this section,
following [HM;
Section~4],
[Mo],
[AMM],
we recall the definition 
and some properties 
of triple $\ell$-th power 
residue symbols for $\ell=2,
3$.

\subsubsection{\bf \boldmath Case of $\ell=2$}\label{2}

Let $p_1,
p_2$ be distinct prime numbers
which satisfy
\begin{equation}\label{pc2}
p_i  \equiv 1~{\rm mod}~4~(i=1,2),~~
\left(\dfrac{p_i}{p_j}\right)
=1~\qquad~(1 \leq i \neq j \leq 2).
\end{equation}
By (\ref{pc2}),
there exist integers 
$x, y, w$ satisfying 
the following conditions 
[Am;~Lemma~1.1]:
\begin{equation}\label{xyw2}
x^2-p_1y^2-p_2w^2=0,
\end{equation}
\[{\rm gcd}(x, y, w)=1,
~~y \equiv 0~{\rm mod}~2,
~~x-y \equiv 1~{\rm mod}~4.\]
Note that the triple $(x,y,w)$ is not unique.
For such a pair $(x,y)$,
we let
\begin{equation}\label{th2}
\theta^{(2)}_{p_1,p_2} := x+\sqrt{p_1}y.
\end{equation}
Moreover,
we set
\begin{equation}\label{R2}
R^{(2)}\left(
=R_{p_1,p_2}^{(2)}\right)
:=\bQ(\sqrt{p_1}, \sqrt{p_2}, 
\sqrt{\theta_{p_1,p_2}^{(2)}}) \subset \bC,
\end{equation}
\begin{equation}\label{K2}
K^{(2)}\left(=K_{p_1,p_2}^{(2)}\right)
:=\bQ(\sqrt{p_1}, \sqrt{p_2}).
\end{equation}

\begin{thm}[{\rm}${\rm [Am; Theorem~1.2, 
Corollary~1.5]}$]\label{lemp2}
The field $R^{(2)}$ is 
a finite Galois extension
of $\bQ$ in $\bC$ 
which satisfies 
the following properties:\\
\begin{inparaenum}[(i)]
 \item The 
Galois group ${\rm Gal}(R^{(2)}/\bQ)$ is 
isomorphic to the Heisenberg group
\[H_3(\bZ/2\bZ):=\left\{ \left( \begin{array}{ccc} 1 & * & * \\ 0 & 1 & * \\ 0 & 0 & 1 \end{array} \right) \Big| \; * \in \bZ/2\bZ \; \right\}\]
(Note that this group $H_3(\bZ/2\bZ)$ 
is isomorphic 
to the dihedral group $D_8$ of order $8$); \\
 \item Prime numbers ramified 
in $R^{(2)}/\bQ$ are only $p_1,
p_2$ with ramification index $2$; \\
 \item The field $R^{(2)}$ is independent 
of the choice of the triple $(x,y,w)$. 
Hence,
$R^{(2)}/\bQ$ depends 
only on the pair $\{p_1, p_2\}$.
\end{inparaenum}
\end{thm}

\begin{thm}[An arithmetic characterization 
of $R^{(2)}${\rm ;~}${\rm [Am; Theorem~2.1]}$]\label{charp2}
Let $p_1,
p_2$ be distinct prime numbers
satisfying (\ref{pc2}).
For a number field $L \subset \bC$, the following conditions are
equivalent:\\
\begin{inparaenum}[(1)]
 \item $L$ is the field  $R^{(2)}$; \\
 \item $L/\bQ$ is a Galois extension in which 
only prime numbers $p_1, p_2$ are ramified 
with ramification index $2$ and whose Galois group is 
isomorphic to the Heisenberg group
$H_3(\bZ/2\bZ)$.
\end{inparaenum}
\end{thm}

Here
we take another prime number 
$p_3$ satisfying
\begin{equation}\label{p_32}
p_3 \equiv 1~{\rm mod}~4,~~
\left(\dfrac{p_i}{p_j}\right)
=1~\qquad~(1 \leq i \neq j \leq 3).
\end{equation}
Note that $p_3$ 
is unramified in $R^{(2)}/\bQ$ 
by Theorem~\ref{lemp2}~(ii).
Then
we introduce an arithmetic symbol 
which controls
the decomposition 
of $p_3$ in the nilpotent extension $R^{(2)}/\bQ$.

\begin{definition}[Triple quadratic residue symbol{\rm ;~}
${\rm [Mo;~Section~8.4]}$]\label{s2}
For prime ideals $(p_1),(p_2),(p_3)$ of $\bZ$ 
where prime numbers $p_1,p_2,p_3$
satisfy (\ref{pc2}) and (\ref{p_32}),
the {\it triple quadratic residue symbol} 
is defined by 
\[[(p_1),(p_2),(p_3)]_2
:=(-1)^{\mu_2(123)} \in \{1,-1\},\]
where 
$\mu_2(123) \in \bZ/2\bZ$ 
is the mod 2 Milnor invariant 
for the prime numbers $p_1,p_2,p_3$.
See [Mo;~Section~8.4] 
for the detailed account 
of $\mu_2(123)$.
\end{definition}

Let $\tilde{\p}_i$ be a prime ideal of $K^{(2)}$ 
above $p_i$.
By (\ref{pc2}) and (\ref{p_32}),
the primes $p_1,
p_2,
p_3$ are completely decomposed or ramified 
in $K^{(2)}/\bQ$
as follows.

\vspace{0.5cm}
\begin{center}
\begin{tikzpicture}
\draw (0,0.5) -- (0,1.5);
\draw (0,2.5) -- (0,3.5);
\draw (0,4.5) -- (0,5.5);
\draw (4,0.5) -- (4,1.5);
\draw (4,2.5) -- (4,3.5);
\draw (7.5,0.5) -- (7.5,1.5);
\draw (7.5,2.5) -- (7.5,3.5);
\draw (11,0.5) -- (11,3.5);
\node at (0,0) {$\bQ$};
\node at (0,2) {$\bQ(\sqrt{p_1})$};
\node at (0,4) {$K^{(2)}$};
\node at (0.57,3.9) {$=$};
\node at (1.9,3.9) {$\bQ(\sqrt{p_1},\sqrt{p_2})$};
\node at (0,6) {$R^{(2)}$};
\node at (0.57,5.9) {$=$};
\node at (2,6) {$K^{(2)}(\sqrt{\theta^{(2)}_{p_1,p_2}})$};
\node at (4,0) {$p_1$};
\node at (4.45,2) {$\tilde{\p}_1 \cap \bQ(\sqrt{p_1})$};
\node at (4,4) {$\tilde{\p}_1$};
\node at (5,1) {ramified};
\node at (5.17,3.18) {completely};
\node at (5.28,2.82) {decomposed};
\node at (7.5,0) {$p_2$};
\node at (7.95,2) {$\tilde{\p}_2 \cap \bQ(\sqrt{p_1})$};
\node at (7.5,4) {$\tilde{\p}_2$};
\node at (8.67,1.18) {completely};
\node at (8.78,0.82) {decomposed};
\node at (8.5,3) {ramified};
\node at (11,0) {$p_3$};
\node at (11,4) {$\tilde{\p}_3$};
\node at (12.17,2.18) {completely};
\node at (12.28,1.82) {decomposed};
\end{tikzpicture}
\end{center}
\vspace{0.5cm}

\begin{thm}[{\rm{[Mo;~Section~8.4,
~Theorem~8.25]}}]\label{22}
Let 
$\sigma_{p_3}:={\rm Frob}_{\tilde{\p}_3} 
\in {\rm Gal}(R^{(2)}/K^{(2)})$ 
be the Frobenius substitution of
$\tilde{\p}_3$ in $R^{(2)}/K^{(2)}$.
Then we have
\[[(p_1),(p_2),(p_3)]_2=\dfrac{\sigma_{p_3}
(\sqrt{\theta^{(2)}_{p_1,p_2}})}{\sqrt{\theta^{(2)}_{p_1,p_2}}}.\]
In particular,
$[(p_1),(p_2),(p_3)]_2=1$ if and only if $p_3$ is 
completely decomposed in $R^{(2)}/\bQ$.
\end{thm}

\begin{rem}
The right side of the equation in Theorem \ref{22} 
is the {\it R${\acute{e}}$dei symbol}
introduced by {L. R${\acute{e}}$dei} in
{\rm [R$\rm{\acute{e}}$]}:
\[[p_1,p_2,p_3]_{\rm R\acute{e}dei}:=\dfrac{\sigma_{p_3}
(\sqrt{\theta^{(2)}_{p_1,p_2}})}{\sqrt{\theta^{(2)}_{p_1,p_2}}}.\]
That is,
Theorem \ref{22} means the triple quadratic residue symbol 
is equal to
the {\it R${\acute{e}}$dei symbol}:
\[[(p_1),(p_2),(p_3)]_2=[p_1,p_2,p_3]_{\rm R\acute{e}dei}.\]
\end{rem}

In {\rm [R$\rm{\acute{e}}$]},
L. R${\rm \acute{e}}$dei proved 
the following reciprocity law 
of the triple symbol.
In [Am],
F. Aamano gave another simple proof of it.

\begin{thm}[Reciprocity law of triple quadratic residue symbols{\rm ;~}
${\rm [R\rm{\acute{e}}],~[Am]}$]
Let $\rho \in S_3$ be any permutation of the set $\{1,
2,
3\}$.
Then
\[[(p_1),(p_2),(p_3)]_2 \cdot [(p_{\rho(1)}),(p_{\rho(2)}),(p_{\rho(3)})]_2
=1,\]
that is
$[(p_1),(p_2),(p_3)]_2=[(p_{\rho(1)}),(p_{\rho(2)}),(p_{\rho(3)})]_2$.
\end{thm}


\subsubsection{\bf \boldmath Case of $\ell=3$}\label{3}
In this section, we essentially follow [HM,~Section~4.2] 
for various assumptions.
Let
$k:=\bQ(\zeta_{3})=\bQ(\sqrt{-3})$
be the Eisenstein field where $\zeta_3:=\exp(\frac{2\pi\sqrt{-1}}{3})
=\frac{-1+\sqrt{-3}}{2}$.
Let
$\p_i=(p_i)~(i=1,2)$ 
be
distinct prime ideals 
of $k$ which satisfy
\begin{equation}\label{pc3}
{\rm N}\p_i  \equiv 1~{\rm mod}~9~(i=1,2),~~
\left(\dfrac{p_i}{p_j}\right)_3
=1~\qquad~(1 \leq i \neq j \leq 2).
\end{equation}
Following [AMM,~Corollary~5.9],
[HM,~Section~4.2],
we assume that
\begin{equation}\label{assume3}
{\rm each~}p_i {\rm ~is~an~associate~of~a~rational~prime~number~in}~k.
\end{equation}
There is an ambiguity of the choice 
of $p_i$ up to 
units $\bZ[\zeta_3]^{\times}
=\{\pm \zeta_3^m \mid m=0,1,2\}$,
but we can take it uniquely 
by posing the following 
condition (cf. [AMM;~Lemma~1.1]):
\begin{equation}\label{pcc3}
p_i \equiv 1~{\rm mod}~(3\sqrt{-3}).
\end{equation}
We fix such a prime element $p_i \in \bZ[\zeta_3]$.
We set $K_i:=k(\sqrt[3]{p_i})$.
The field $K_i$ is a cyclic extension 
of degree $3$ over $k$
in which only $\p_i$ is ramified
(cf. [AMM;~Theorem~3.5]). 
Let $\phi$ be a generator 
of $\Gal(K_1/k)$ determined by 
$\phi(\sqrt[3]{p_1})=\zeta_3\sqrt[3]{p_1}$.
By (\ref{pc3}) and (\ref{pcc3}),
there exist algebraic integers
\begin{equation}\label{alpha}
\alpha_{p_1,p_2} \in {\mathcal O}_{K_1},
~\qquad~
w \in \bZ[\zeta_3]
\end{equation}
together with prime ideals ${\mathfrak P}, 
{\mathfrak B}$ of $K_1$
which satisfy the following conditions
[AMM;~Proposition~5.6]:
\begin{equation}\label{norm}
N_{K_1/k}(\alpha_{p_1,p_2})=p_2w^3,
\end{equation}
\[(\alpha_{p_1,p_2})={\mathfrak P}
^e{\mathfrak B}^f,~~
(e,3)=1,~
({\mathfrak B}, 3)=1,~
f \equiv 0~{\rm mod}~3.\]
Note that $\alpha_{p_1,p_2}$ is not unique.
For such an $\alpha_{p_1,p_2} 
\in {\mathcal O}_{K_1}$,
we let
\begin{equation}\label{oth3}
\theta^{(3)}_{p_1,p_2}
:=\phi(\alpha_{p_1,p_2})(\phi^2(\alpha_{p_1,p_2}))^2.
\end{equation}
Moreover,
we set 
\begin{equation}\label{K3}
R^{(3)}\left(=R_{p_1,p_2}^{(3)}\right)
:=k(\sqrt[3]{p_1}, \sqrt[3]{p_2}, 
\sqrt[3]{\theta^{(3)}_{p_1,p_2}}) \subset \bC,
\end{equation}
\begin{equation}\label{K3}
K^{(3)}\left(=K_{p_1,p_2}^{(3)}\right)
:=k(\sqrt[3]{p_1}, \sqrt[3]{p_2}).
\end{equation}

By using the assumption (\ref{assume3}),
we obtain the following theorem.

\begin{thm}[{\rm}${\rm [AMM;~Theorem~5.11,
~Corollary~5.12]}$]\label{lemp3}
The field $R^{(3)}$ is 
a finite Galois extension
of $k$ in $\bC$ 
which holds the following properties:\\
\begin{inparaenum}[(i)]
 \item The Galois group 
${\rm Gal}(R^{(3)}/k)$ is isomorphic to
the Heisenberg group
\[H_3(\bZ/3\bZ):=\left\{ \left( \begin{array}{ccc} 1 & * & * \\ 0 & 1 & * \\ 0 & 0 & 1 \end{array} \right) \Big| \; * \in \bZ/3\bZ \; \right\};\]
\\
 \item Prime ideals ramified 
in $R^{(3)}/k$ are only $\p_1,
\p_2$ with ramification index $3$; \\
 \item The field $R^{(3)}$ is independent 
of the choice of 
$\alpha_{p_1,p_2} \in {\mathcal O}_{K_1}$.
Hence,
$R^{(3)}/k$ depends only 
on the pair $\{\p_1, \p_2\}$.
\end{inparaenum}
\end{thm}

\begin{thm}[An arithmetic characterization 
of $R^{(3)}${\rm ;~}${\rm cf.~[AMM; Corollary~5.12]}$]\label{charp3}
Let $\p_1=(p_1),
\p_2=(p_2)$ be distinct prime ideals of $k$
satisfying (\ref{pc3}),
(\ref{assume3}),
(\ref{pcc3}) and (\ref{p_33}).
For a finite extension field $L$ of $k$ in $\bC$, the following conditions are
equivalent:\\
\begin{inparaenum}[(1)]
 \item $L$ is the field  $R^{(3)}$; \\
 \item $L/k$ is a Galois extension 
in which only primes $\p_1,
\p_2$ are ramified with ramification index $3$ 
and whose Galois group is isomorphic to
the Heisenberg group
$H_3(\bZ/3\bZ)$.
\end{inparaenum}
\end{thm}

Here
we take another prime ideal 
$\p_3=(p_3)$ of $k$ satisfying
\begin{equation}\label{p_33}
{\rm N}\p_3  \equiv 1~{\rm mod}~9,~~
\left(\dfrac{p_i}{p_j}\right)_3
=1~~(1 \leq i \neq j \leq 3).
\end{equation}
Note that $\p_3$ 
is unramified in $R^{(3)}/k$ 
by Theorem~\ref{lemp3}~(ii).
Then
we introduce an arithmetic symbol 
which controls
the decomposition of $\p_3$ 
in the nilpotent extension $R^{(3)}/k$.

\begin{definition}[Triple cubic residue symbol
{\rm ;~}${\rm [AMM;~Definition~6.2]}$]\label{s3}
For a triple of primes $(\p_1,\p_2,\p_3)$ 
of $k$ satisfying (\ref{pc3}) and (\ref{p_33}),
the {\it triple cubic residue symbol} 
is defined by 
\[[\p_1,\p_2,\p_3]_3
:=\zeta_3^{\mu_3(123)} 
\in \{1,\zeta_3,\zeta_3^{-1}\},\]
where 
$\mu_3(123) \in \bZ/3\bZ$ is 
the mod 3 Milnor invariant 
for the primes $\p_1,\p_2,\p_3$.
See [AMM;~(2.3) of Chapter~2, Theorem~4.4] 
for the detailed account 
of $\mu_3(123)$.
\end{definition}

Let $\tilde{\p}_i$ be a prime ideal 
of $K^{(3)}$ above $\p_i$.
By (\ref{pc3}) and (\ref{p_33}),
the primes $\p_1,
\p_2,
\p_3$ are completely decomposed or
ramified in $K^{(3)}/k$ as follows.

\vspace{0.5cm}
\begin{center}
\begin{tikzpicture}
\draw (0,2.5) -- (0,3.5);
\draw (0,4.5) -- (0,5.5);
\draw (0,6.5) -- (0,7.5);
\draw (4,2.5) -- (4,3.5);
\draw (4,4.5) -- (4,5.5);
\draw (7.5,2.5) -- (7.5,3.5);
\draw (7.5,4.5) -- (7.5,5.5);
\draw (11,2.5) -- (11,5.5);
\node at (0,2) {$k$};
\node at (0.35,1.95) {$=$};
\node at (1.1, 1.98) {$\bQ(\zeta_3)$};
\node at (0,4) {$K_1$};
\node at (0.57,3.95) {$=$};
\node at (1.5,4) {$k(\sqrt[3]{p_1})$};
\node at (0,6) {$K^{(2)}$};
\node at (0.57,5.9) {$=$};
\node at (1.9,6) {$k(\sqrt[3]{p_1}, \sqrt[3]{p_2})$};
\node at (0, 8) {$R^{(3)}$};
\node at (0.57, 7.9) {$=$};
\node at (1.9, 8) {$K^{(3)}(\sqrt[3]{\theta_{p_1, p_2}^{(3)}})$};
\node at (4,2) {$\p_1$};
\node at (4,4) {$\tilde{\p}_1 \cap K_1$};
\node at (4,6) {$\tilde{\p}_1$};
\node at (5,3) {ramified};
\node at (5.17,5.18) {completely};
\node at (5.28,4.82) {decomposed};
\node at (7.5,2) {$\p_2$};
\node at (7.5,4) {$\tilde{\p}_2 \cap K_1$};
\node at (7.5,6) {$\tilde{\p}_2$};
\node at (8.67,3.18) {completely};
\node at (8.78,2.82) {decomposed};
\node at (8.5,5) {ramified};
\node at (11,2) {$\p_3$};
\node at (11,6) {$\tilde{\p}_3$};
\node at (12.17,4.18) {completely};
\node at (12.28,3.82) {decomposed};
\end{tikzpicture}
\end{center}
\vspace{0.5cm}

\begin{thm}[{\rm{[AMM;~Theorem.6.3]}}]\label{222}
Let
$\sigma_{\p_3}:={\rm Frob}_{\tilde{\p}_3} 
\in {\rm Gal}(R^{(3)}/K^{(3)})$ 
be the Frobenius substitution 
of $\tilde{\p}_3$ in $R^{(3)}/K^{(3)}$.
Then we have
\[[\p_1,\p_2,\p_3]_3
=\dfrac{\sigma_{\p_3}(\sqrt[3]{\theta^{(3)}_{p_1,p_2}})}
{\sqrt[3]{\theta^{(3)}_{p_1,p_2}}}.\]
In particular,
$[\p_1,\p_2,\p_3]_3=1$ 
if and only if $\p_3$ is 
completely decomposed in $R^{(3)}/k$.
\end{thm}

\begin{thm}[a reciprocity law of triple cubic residue symbols{\rm ;~}
${\rm [AMM;~Proposition~6.5]}$]\label{recip3}
We have
\[[\p_1,\p_2,\p_3]_3 
\cdot [\p_2,\p_1,\p_3]_3=1,\]
that is $[\p_2,\p_1,\p_3]_3 
= [\p_1,\p_2,\p_3]_3^{-1}$.
\end{thm}

\vspace{3mm}

Hereafter,
following Hirano-Morishita [HM,~Section~4.2],
we shall restrict ourselves 
to 
the case with 
\vspace{3mm}
\begin{center}
\underline{Assumption~(A)}~:
~~The pair
$(p_1,p_2)$
has
$\alpha_{p_1,p_2} \in {\mathcal O}_{K_1}$ 
in (\ref{alpha}) and (\ref{norm}) which
can be given in the form
$\alpha_{p_1,p_2}=x+y\sqrt[3]{p_1}
~\qquad~(x,y \in k).$\\
\end{center}
\vspace{3mm}
Under the above assumption~(A),
the conditions (\ref{norm}), (\ref{oth3})
are equivalent to
\begin{equation}\label{xyw3}
x^3+p_1 y^3=p_2 w^3,
\end{equation}
\begin{equation}\label{th3}
\theta^{(3)}_{p_1,p_2} =(x+\zeta_3 y \sqrt[3]{p_1})
(x+\zeta_3^2 y \sqrt[3]{p_1})^2.
\end{equation}
These equations will play an important role in the next section.


\section{Triple $\ell$-th power residue symbols 
and $\ell$-adic Galois polylogarithms}

In this section,
we interpret 
triple $\ell$-th power residue symbols 
in terms of $\ell$-adic 
Galois polylogarithms for $\ell = 2,
3$.
As a result,
we derive a reciprocity law of 
triple $\ell$-th power residue symbols 
from a functinal equation 
of $\ell$-adic Galois polylogarithms.

\subsection{\bf \boldmath Main formula}

Let $\ell \in \{2,
3\}$.
Let $k := 
\left\{
\begin{array}{ll}
    \mathbb{Q} & ({\rm if}~\ell=2),\\
   \mathbb{Q}(\zeta_3) & ({\rm if}~\ell=3)
  \end{array}
  \right.$ and 
\begin{equation}\label{Kl}
K :=
\bQ(\zeta_\ell)
(\sqrt[\ell]{p_1}, \sqrt[\ell]{p_2}),
\end{equation}
where $\zeta_{\ell}
=\exp(\frac{2\pi\sqrt{-1}}{\ell})$
is a fixed primitive $\ell$-th root of unity in $\overline{K} \subset \bC$.
We set
\[p_i \in \bZ[\zeta_\ell]
~\qquad~(i=1,2,3),
~\qquad~x, y, w \in k,\]
\[\theta^{(\ell)}_{p_1,p_2}, R^{(\ell)}_{p_1,p_2}, K^{(\ell)}_{p_1,p_2}\]
as in Section~\ref{2} for $\ell=2$ and 
as in Section~\ref{3} 
with the assumption (A) for $\ell=3$.
Note that $K=K^{(\ell)}_{p_1,p_2}$ by (\ref{K2}) and (\ref{K3}).

Hence,
by (\ref{th2}), (\ref{th3}) 
and (\ref{xyw2}), (\ref{xyw3}),
we have
\begin{equation}\label{xywl}
x^{\ell}-(-y)^{\ell} p_1 = w^\ell p_2,
\end{equation}
\begin{equation}\label{thl}
\theta^{(\ell)}_{p_1,p_2}=\prod_{i=0}^{\ell -1}
(x+\zeta_{\ell}^{i}y\sqrt[\ell]{p_1})^{i}.
\end{equation}
For the prime element 
$p_i \in \bZ[\zeta_\ell]~(i=1,2,3)$, 
we denote by
\[\p_i = (p_i)\]
the prime ideal 
of $k$ generated by $p_i$.
For the triple 
$(\p_1, \p_2, \p_3)$ of primes of $k$,
the triple $\ell$-th power 
residue symbol $[\p_1, \p_2, \p_3]_\ell$ 
is defined
as discussed 
in Section~\ref{triplesymbol}.

Moreover,
we choose
\begin{equation}\label{rz}
z:=p_1\left(-\dfrac{y}{x}\right)^\ell.
\end{equation}
Since $z \in K \backslash \{0,1\}$,
we regard $z$ 
as a $K$-rational point 
of ${\mathbb P}_K
^1\backslash \{0,1,\infty\}$.
Let
\begin{equation}\label{sigma}
\tilde{\sigma}_{\p_3} \in \Gal( \overline{K}/K)
\end{equation}
be an extension 
of the Frobenius substitution $\sigma_{\p_3} := 
{\rm Frob}_{{\tilde{\p}}_3} \in \Gal(R_{p_1,p_2}^{(\ell)}/K)$ 
where $\tilde{\p}_3$ is a prime ideal 
of $K$ above $\p_3$.
Let $\bar{z}: \Spec \overline{K} 
\to {\mathbb P}^1_{\overline{K}}
\backslash \{0,1,\infty\}$ be 
the base change of $z$ 
via $\Spec \overline{K} \to \Spec K$.
Fix a homotopy class 
\[\gamma \in 
\pi_1^{\top}
({\mathbb P}^1(\bC)
\backslash \{0,1,\infty\}; \overrightarrow{01}, z)\]
of a piece-wise smooth topological 
path on ${\mathbb P}^1(\bC)
\backslash \{0,1,\infty\}$ from $\overrightarrow{01}$ to $z$.
Then,
the 2nd 
$\ell$-adic Galois polylogarithms
$\ell i_{2}^{(\ell)}(z,\gamma): G_K \to \bQ_\ell$,
$\tilde{\chi}_2^{z, \gamma}: G_K \to \bZ_\ell$
are defined
as discussed 
in Section~\ref{lgp}.


\begin{prop}\label{depn}
Let the notations 
and assumptions be as above.
For any $\tau \in G_K$,
the value
$\ell i_{2}^{(\ell)}(z,\gamma)
(\tau)~{\rm mod}~\ell$,
together with
${\tilde{\chi}_{2}
^{z, \gamma}(\tau)}~{\rm mod}~\ell$,
is independent 
of the choice of $\gamma$.
\end{prop}

\begin{proof}
Let
$\tau \in G_K$.
By (\ref{Kl}),
we have $\chi(\tau) \equiv 1,
\rho_{z, \gamma}(\tau) \equiv 0~{\rm mod}~\ell$.
Hence,
it follows from Definition~\ref{char} that
\begin{equation}\label{value}
\zeta_{\ell}^{\tilde{\chi}_{2}
^{z, \gamma}(\tau)} 
= \tau
\left(\displaystyle 
\prod_{i=0}^{\ell -1}
(1-\zeta_{\ell}^{i}z^{1/\ell})
^{\frac{i}{\ell}}\right) 
\bigg/ \displaystyle ~\prod_{i=0}
^{\ell -1}(1-\zeta_{\ell}
^{i}z^{1/\ell})^{\frac{i}{\ell}}.
\end{equation}
Let $\gamma_0,
\gamma_1 \in \pi_1^{\top}
({\mathbb P}^1(\bC)
\backslash \{0,1,\infty\}; \overrightarrow{01}, z)$.
For $\epsilon \in \{0,
1\}$,
we choose $z_{\epsilon}^{1/\ell}$,
$(1-\zeta_{\ell}^{i}z_{\epsilon}^{1/\ell})$
as the specific $\ell$-th power roots determined 
by $\gamma_{\epsilon}$~(cf.~[NW1]).

First,
in order to prove that ${\tilde{\chi}_{2}
^{z, \gamma}(\tau)}~{\rm mod}~\ell$
is independent 
of the choice of $\gamma$,
it suffices to show
\begin{equation}\label{eq01}
\zeta_{\ell}^{{\tilde{\chi}_{2}
^{z, \gamma_0}(\tau)}}=\zeta_{\ell}^{{\tilde{\chi}_{2}
^{z, \gamma_1}(\tau)}}
\end{equation}
by comparing the right hand side of (\ref{value}) for $\gamma=\gamma_0,
\gamma_1$.
We now show (\ref{eq01}).
Let \[A_{\epsilon}:=\prod_{i=0}^{\ell -1}
(1-\zeta_{\ell}^{i}z_{\epsilon}^{1/\ell})
^{\frac{i}{\ell}}~(\epsilon \in \{0,
1\}).\]
Take $s \in \bZ/\ell \bZ$ such that
$
z_{1}^{1/\ell}=\zeta_{\ell}^{-s} \cdot z_{0}^{1/\ell}$.
For each $i \in \bZ/\ell \bZ$,
there exists $t_i \in \bZ/\ell \bZ$ such that
$
(1-\zeta_{\ell}^{i}z_{1}^{1/\ell})^{\frac{1}{\ell}}
=\zeta_{\ell}^{t_i}(1-\zeta_{\ell}^{i-s}z_{0}^{1/\ell})^{\frac{1}{\ell}}$
since $(1-\zeta_{\ell}^{i}z_{1}^{1/\ell})=(1-\zeta_{\ell}^{i-s}z_{0}^{1/\ell})$.
Then,
we compute $A_1/A_0$ as follows:
\begin{align*}
\dfrac{A_1}{A_0} & 
= \dfrac{\displaystyle \prod_{i=0}^{\ell -1}
(1-\zeta_{\ell}^{i}z_{0}^{1/\ell})
^{\frac{i}{\ell}} \cdot {\zeta_{\ell}}^{\sum\limits_{i=0}^{\ell-1}it_i}}{\displaystyle \prod_{i=0}^{\ell -1}
(1-\zeta_{\ell}^{i}z_{0}^{1/\ell})
^{\frac{i}{\ell}}} \\
& = \dfrac{\displaystyle \prod_{j=0}^{\ell -1}
(1-\zeta_{\ell}^{j}z_{0}^{1/\ell})
^{\frac{j+s}{\ell}}}{\displaystyle \prod_{i=0}^{\ell -1}
(1-\zeta_{\ell}^{i}z_{0}^{1/\ell})
^{\frac{i}{\ell}}} \cdot \zeta_{\ell}^{\sum\limits_{i=0}^{\ell-1}it_i} \\
& = \left({\displaystyle \prod_{j=0}^{\ell -1}
(1-\zeta_{\ell}^{j}z_{0}^{1/\ell})
^{\frac{1}{\ell}}}\right)^{s} \cdot \zeta_{\ell}^{\sum\limits_{i=0}^{\ell-1}it_i}.
\end{align*}
Since
${\displaystyle \prod_{j=0}^{\ell -1}
(1-\zeta_{\ell}^{j}z_{0}^{1/\ell})}=1-z={\left(\dfrac{w}{x}\right)}^{\ell}p_2$ 
and $K =
\bQ(\zeta_\ell)
(\sqrt[\ell]{p_1}, \sqrt[\ell]{p_2})$
by (\ref{Kl}), (\ref{xywl}), and (\ref{rz}),
we obtain $\dfrac{A_1}{A_0} \in K$,
hence
\[
\zeta_{\ell}^{{\tilde{\chi}_{2}
^{z, \gamma_0}(\tau)}}
=\dfrac{\tau(A_0)}{A_0}
=\dfrac{\tau(A_1)}{A_1}
=\zeta_{\ell}^{{\tilde{\chi}_{2}
^{z, \gamma_1}(\tau)}}.
\]
This complete the proof of (\ref{eq01}).

Moreover,
by Theorem~\ref{explicit} 
and  (\ref{Kl}),
(\ref{sigma}),
we have
\begin{equation}\label{val}
\ell i_{2}^{(\ell)}
(z,\gamma)({\tau})
\equiv-{\tilde{\chi}_{2}
^{z, \gamma}({\tau})}~{\rm mod}~\ell;
\end{equation}
therefore,
$\ell i_{2}^{(\ell)}(z,\gamma)
({\tau})~{\rm mod}~\ell$ 
is also independent 
of the choice of $\gamma$.
\end{proof}

\begin{dfn}
Let the notations 
and assumptions be as above.
Based on Proposition~\ref{depn},
for $\tau \in G_K$,
we let
\[
\ell i_{2}^{(\ell)}
(z)(\tau)~{\rm mod}~\ell
:=
\ell i_{2}^{(\ell)}
(z,\gamma)({\tau})~{\rm mod}~\ell,
\]
\[
\tilde{\chi}_{2}
^{z}(\tau)~{\rm mod}~\ell
:=
\tilde{\chi}_{2}
^{z, \gamma}(\tau)~{\rm mod}~\ell,\]
that is
\[
\zeta_{\ell}
^{\ell i_{2}^{(\ell)}(z)(\tau)}
:=
\zeta_{\ell}
^{\ell i_{2}^{(\ell)}(z, \gamma)(\tau)},~~
\zeta_{\ell}
^{\tilde{\chi}_{2}
^{z}
(\tau)}
:=
\zeta_{\ell}
^{\tilde{\chi}_{2}
^{z, \gamma}
(\tau)}.
\]
\end{dfn}

Now
we shall describe the triple symbol 
$[\p_1,\p_2,\p_3]_{\ell}$ 
by a special value of the 2nd $\ell$-adic 
Galois polylogarithm.

\begin{thm}\label{tl}
Let the notations 
and assumptions be as above.
For $\ell \in \{2,
3\}$,
we have
\begin{align*}
[\p_1,\p_2,\p_3]_{\ell} & 
= \dfrac{\tilde{\sigma}_{\p_3}
(x^{\frac{1}{2}(\ell-1)})}
{x^{\frac{1}{2}(\ell-1)}}
\cdot \zeta_{\ell}
^{\tilde{\chi}_{2}
^{z}
(\tilde{\sigma}_{\p_3})}\\
& = \dfrac{\tilde{\sigma}_{\p_3}
(x^{\frac{1}{2}(\ell-1)})}
{x^{\frac{1}{2}(\ell-1)}}
\cdot \zeta_{\ell}
^{-\ell i_{2}^{(\ell)}
(z)(\tilde{\sigma}_{\p_3})}.
\end{align*}
\end{thm}

\begin{proof}
Let $\ell \in \{2,
3\}$.
We compute the triple symbol 
$[\p_1,\p_2,\p_3]_{\ell}$ as follows:
\begin{align*}
  [\p_1,\p_2,\p_3]_{\ell} 
  & = \sigma_{\p_3} \left(\sqrt[\ell]
{\theta^{(\ell)}_{p_1,p_2}}\right)/\sqrt[\ell]
{\theta^{(\ell)}_{p_1,p_2}}~~
({\rm by}~~{\rm Theorem}~\ref{22},
~{\rm Theorem}~\ref{222})\\
  & = \tilde{\sigma}_{\p_3} \left(\sqrt[\ell]
{\theta^{(\ell)}_{p_1,p_2}}\right)
/\sqrt[\ell]{\theta^{(\ell)}_{p_1,p_2}}  \\
  & = \tilde{\sigma}_{\p_3}
\left(\prod_{i=0}^{\ell -1}
(x+\zeta_{\ell}^{i}y\sqrt[\ell]{p_1})
^{\frac{i}{\ell}}\right) 
\bigg/ \prod_{i=0}^{\ell -1}
(x+\zeta_{\ell}^{i}y\sqrt[\ell]{p_1})
^{\frac{i}{\ell}}~~({\rm by}
~~(\ref{thl}))\\
   & =  \dfrac{\tilde{\sigma}_{\p_3}
\left(\displaystyle \prod_{i=0}
^{\ell -1}x
^{\frac{i}{\ell}}\right)}
{\displaystyle ~\prod_{i=0}
^{\ell -1}x^{\frac{i}{\ell}}} 
\cdot \dfrac{\tilde{\sigma}_{\p_3}
\left(\displaystyle \prod_{i=0}
^{\ell -1}
\left(1+\zeta_{\ell}
^{i}\frac{y}{x}p_1
^{1/\ell}\right)
^{\frac{i}{\ell}}\right)}
{\displaystyle ~\prod_{i=0}
^{\ell -1}
(1+\zeta_{\ell}^{i}\frac{y}{x}p_1
^{1/\ell})^{\frac{i}{\ell}}} \\
   & =  \dfrac{\tilde{\sigma}_{\p_3}
(x^{\frac{1}{2}(\ell-1)})}
{x^{\frac{1}{2}(\ell-1)}} 
\cdot \dfrac{\tilde{\sigma}_{\p_3}
\left(\displaystyle \prod_{i=0}
^{\ell -1}
\left(1+\zeta_{\ell}
^{i}\frac{y}{x}p_1
^{1/\ell}\right)
^{\frac{i}{\ell}}\right)}
{\displaystyle ~\prod_{i=0}
^{\ell -1}
(1+\zeta_{\ell}
^{i}\frac{y}{x}p_1
^{1/\ell})^{\frac{i}{\ell}}}.
\end{align*}
Since
$z
=p_1\left(-\dfrac{y}{x}\right)^{\ell}$
by (\ref{rz}),
the second factor 
of the above last side is equal to
\[
\dfrac{\tilde{\sigma}_{\p_3}
\left(\displaystyle \prod_{i=0}
^{\ell -1}
(1-\zeta_{\ell}^{i}z^{1/\ell})
^{\frac{i}{\ell}}\right)}
{\displaystyle ~\prod_{i=0}
^{\ell -1}
(1-\zeta_{\ell}^{i}z^{1/\ell})
^{\frac{i}{\ell}}}\\
= \zeta_{\ell}^{\tilde{\chi}_{2}
^{z}(\tilde{\sigma}_{\p_3})}
~~({\rm by}~~(\ref{value})).\]
Therefore, 
by combining above formulas and (\ref{val}), 
we obtain
\[
[\p_1,\p_2,\p_3]_{\ell} 
= \dfrac{\tilde{\sigma}_{\p_3}
(x^{\frac{1}{2}(\ell-1)})}
{x^{\frac{1}{2}(\ell-1)}}
\cdot \zeta_{\ell}
^{\tilde{\chi}_{2}
^{z}(\tilde{\sigma}_{\p_3})}\\
= \dfrac{\tilde{\sigma}_{\p_3}
(x^{\frac{1}{2}(\ell-1)})}
{x^{\frac{1}{2}(\ell-1)}}
\cdot \zeta_{\ell}
^{-\ell i_{2}^{(\ell)}
(z)(\tilde{\sigma}_{\p_3})}.
\]
\end{proof}


\begin{cor}[Case of $\ell=2$]\label{c2}
Let the notations 
and assumptions be as above.
Then we have
\[
[\p_1,\p_2,\p_3]_2 
= (-1)^{\rho_{x}(\tilde{\sigma}_{\p_3})
-\ell i_{2}^{(2)}
(z)(\tilde{\sigma}_{\p_3})},
\]
where the value 
$\rho_{x}(\tilde{\sigma}_{\p_3}) 
\in \bZ/2\bZ$ is defined by
$\tilde{\sigma}_{\p_3}(\sqrt{x})/\sqrt{x}
=(-1)^{\rho_{x}(\tilde{\sigma}_{\p_3})}$.
Hence,
we obtain
\[\mu_2(123)
=\rho_{x}(\tilde{\sigma}_{\p_3})
-\ell i_{2}^{(2)}(z)
(\tilde{\sigma}_{\p_3})~{\rm mod}~2.\]
\end{cor}

\begin{proof}
The assertion follows 
from Theorem~\ref{tl} 
and Definition~\ref{s2}.
\end{proof}

\begin{cor}[Case of $\ell=3$]\label{c3}
Let the notations 
and assumptions be as above.
Then we have
\[
[\p_1,\p_2,\p_3]_3 
= {\zeta_3}^{-\ell i_{2}^{(3)}
(z)(\tilde{\sigma}_{\p_3})}
.\]
Hence,
we obtain
\[\mu_3(123)=-\ell i_{2}^{(3)}
(z)(\tilde{\sigma}_{\p_3}) 
~{\rm mod}~3.\]
\end{cor}

\begin{proof}
The assertion follows 
from Theorem~\ref{tl} 
and Definition~\ref{s3}.
\end{proof}

\subsection{\bf \boldmath Deriving a reciprocity law}

Let the notations and assumptions 
be as in previous section.
Note that $\gamma'=\delta \cdot \varphi(\gamma)  \in \pi_1^{\top}
({\mathbb P}^1(\bC)
\backslash \{0,1,\infty\}; \overrightarrow{01}, 1-z)$ 
is as in Section~1.1.

\begin{prop}\label{depn2}
For any $\tau \in G_K$,
the value
$\ell i_{2}^{(\ell)}(1-z,\gamma')
(\tau)~{\rm mod}~\ell$,
together with
${\tilde{\chi}_{2}
^{1-z, \gamma'}(\tau)}~{\rm mod}~\ell$,
is independent 
of the choice of $\gamma$.
\end{prop}

\begin{proof}
The proof can be done in the same way 
as the proof of Proposition \ref{depn}.
\end{proof}

\begin{dfn}
Based on Proposition~\ref{depn2},
for $\tau \in G_K$,
we let
\[
\ell i_{2}^{(\ell)}
(1-z)(\tau)~{\rm mod}~\ell
:=
\ell i_{2}^{(\ell)}
(1-z,\gamma')({\tau})~{\rm mod}~\ell,
\]
\[
\tilde{\chi}_{2}
^{1-z}(\tau)~{\rm mod}~\ell
:=
\tilde{\chi}_{2}
^{1-z, \gamma'}(\tau)~{\rm mod}~\ell,\]
that is
\[
\zeta_{\ell}
^{\ell i_{2}^{(\ell)}(1-z)(\tau)}
:=
\zeta_{\ell}
^{\ell i_{2}^{(\ell)}(1-z, \gamma')(\tau)},~~
\zeta_{\ell}
^{\tilde{\chi}_{2}
^{1-z}
(\tau)}
:=
\zeta_{\ell}
^{\tilde{\chi}_{2}
^{1-z, \gamma'}
(\tau)}.
\]
\end{dfn}

Firstly,
to derive a reciprocity law 
of triple $\ell$-th 
power residue symbols,
we describe the triple symbol 
$[\p_2,\p_1,\p_3]_{\ell}$ 
by the 2nd $\ell$-adic 
Galois polylogarithmic character.

\begin{thm}\label{aaa}
For $\ell \in \{2,3\}$,
we have
\[[\p_2,\p_1,\p_3]_{\ell} =
\dfrac{\tilde{\sigma}_{\p_3}
(x^{\frac{1}{2}(\ell-1)})}
{x^{\frac{1}{2}(\ell-1)}}\cdot 
\zeta_{\ell}^{\tilde{\chi}_{2}
^{1-z}
(\tilde{\sigma}_{\p_3})}.\]
\end{thm}

\begin{proof}
Let $\ell \in \{2,3\}$.
Since
$x^\ell-(-y)^\ell p_1=w^\ell p_2~~ \Longleftrightarrow~~x^\ell-w^\ell p_2
=(-y)^\ell p_1$ by (\ref{xywl}),
we can take
\begin{equation}\label{theta3}
\theta^{(\ell)}_{p_2,p_1}=\prod_{i=0}^{\ell -1}
(x-\zeta_{\ell}^{i}w\sqrt[\ell]{p_2})
^{i}
\end{equation}
by replacing $p_1$, $p_2$, 
and $y$ in (\ref{thl}) with $p_2$, $p_1$, 
and $-w$.
As with Theorem~\ref{tl},
we compute the triple symbol 
$[\p_2,\p_1,\p_3]_{\ell}$ as follows:
\begin{align*}
  [\p_2,\p_1,\p_3]_{\ell} 
  & = \sigma_{\p_3} \left(\sqrt[\ell]
{\theta^{(\ell)}_{p_2,p_1}}\right)/\sqrt[\ell]
{\theta^{(\ell)}_{p_2,p_1}}~~
({\rm by}~~{\rm Theorem}~\ref{22},
~{\rm Theorem}~\ref{222})\\
  & = \sigma_{\p_3} 
\left(\prod_{i=0}^{\ell -1}
(x-\zeta_{\ell}^{i}w\sqrt[\ell]{p_2})
^{\frac{i}{\ell}}\right) 
\bigg/ \prod_{i=0}^{\ell -1}
(x-\zeta_{\ell}^{i}w\sqrt[\ell]{p_2})
^{\frac{i}{\ell}}~~
({\rm by}~~(\ref{theta3}))\\
  & = \tilde{\sigma}_{\p_3}
\left(\prod_{i=0}^{\ell -1}
(x-\zeta_{\ell}^{i}w\sqrt[\ell]{p_2})
^{\frac{i}{\ell}}\right) 
\bigg/ \prod_{i=0}^{\ell -1}
(x-\zeta_{\ell}^{i}w\sqrt[\ell]{p_2})
^{\frac{i}{\ell}} \\
   & =  \dfrac{\tilde{\sigma}_{\p_3}
\left(\displaystyle \prod_{i=0}
^{\ell -1}x^{\frac{i}{\ell}}\right)}
{\displaystyle ~\prod_{i=0}
^{\ell -1}x^{\frac{i}{\ell}}} 
\cdot \dfrac{\tilde{\sigma}_{\p_3}
\left(\displaystyle \prod_{i=0}
^{\ell -1}\left(1-\zeta_{\ell}
^{i}\frac{w}{x}p_2^{1/\ell}\right)
^{\frac{i}{\ell}}\right)}
{\displaystyle ~\prod_{i=0}^{\ell -1}
\left(1-\zeta_{\ell}^{i}\frac{w}{x}p_2
^{1/\ell}\right)^{\frac{i}{\ell}}} \\
   & =  \dfrac{\tilde{\sigma}_{\p_3}
(x^{\frac{1}{2}(\ell-1)})}
{x^{\frac{1}{2}(\ell-1)}} 
\cdot \dfrac{\tilde{\sigma}_{\p_3}
\left(\displaystyle \prod_{i=0}
^{\ell -1}
\left(1-\zeta_{\ell}^{i}
\frac{w}{x}p_2^{1/\ell}\right)
^{\frac{i}{\ell}}\right)}
{\displaystyle ~\prod_{i=0}^{\ell -1}
\left(1-\zeta_{\ell}^{i}\frac{w}{x}p_2
^{1/\ell}\right)^{\frac{i}{\ell}}}.
\end{align*}
Since $1-z=\dfrac{x
^{\ell}-(-y)^{\ell} p_1}
{x^{\ell}}=\dfrac{w^\ell}
{x^\ell}p_2$ by (\ref{xywl}),
the second factor 
of the above last side is equal to
\[
\dfrac{\tilde{\sigma}_{\p_3}
\left(\displaystyle 
\prod_{i=0}^{\ell -1}
(1-\zeta_{\ell}^{i}{(1-z)}
^{1/\ell})^{\frac{i}{\ell}}\right)}
{\displaystyle ~\prod_{i=0}^{\ell -1}
(1-\zeta_{\ell}^{i}{(1-z)}
^{1/\ell})^{\frac{i}{\ell}}}\\
= \zeta_{\ell}^{\tilde{\chi}_{2}
^{1-z}(\tilde{\sigma}_{\p_3})}
~~({\rm by~~Definition}~\ref{char},~(\ref{Kl}),~(\ref{sigma})).
\]
Therefore
we obtain the assertion of the theorem.
\end{proof}

Now,
we derive a reciprocity law of 
triple $\ell$-th power 
residue symbols 
from the functional equation 
of $\ell$-adic Galois polylogarithms
introduced in Theorem~\ref{func}.

\begin{cor}[a reciprocity law]\label{reci}
Let the notations 
and assumptions be as above.
For $\ell \in \{2,3\}$,
we have
\[[\p_1,\p_2,\p_3]_{\ell} 
\cdot [\p_2,\p_1,\p_3]_{\ell} =1.\]
\end{cor}

\begin{proof}
By combining Theorem~\ref{tl} 
and Theorem~\ref{aaa},
\begin{align*}
[\p_1,\p_2,\p_3]_{\ell} 
\cdot [\p_2,\p_1,\p_3]_{\ell} 
& = \left\{ \begin{array}{ll}
\frac{\tilde{\sigma}_{\p_3}(\sqrt{x})}
{\sqrt{x}} (-1)^{{\tilde{\chi}_{2}
^{z}(\tilde{\sigma}_{\p_3})}} 
\cdot \frac{\tilde{\sigma}_{\p_3}
(\sqrt{x})}{\sqrt{x}} (-1)
^{{\tilde{\chi}_{2}
^{1-z}
(\tilde{\sigma}_{\p_3})}} & ({\rm if}~\ell=2), \\
    {\zeta_3}^{\tilde{\chi}_{2}
^{z}(\tilde{\sigma}_{\p_3})} 
\cdot {\zeta_3}^{\tilde{\chi}_{2}
^{1-z}
(\tilde{\sigma}_{\p_3})} & ({\rm if}~\ell=3)
  \end{array} \right. \\
& = \left\{ \begin{array}{ll}
    (-1)^{{\tilde{\chi}_{2}
^{z}(\tilde{\sigma}_{\p_3})}
+{\tilde{\chi}_{2}^{1-z}
(\tilde{\sigma}_{\p_3})}} & ({\rm if}~\ell=2), \\
    {\zeta_3}^{\tilde{\chi}_{2}
^{z}(\tilde{\sigma}_{\p_3})
+\tilde{\chi}_{2}
^{1-z}(\tilde{\sigma}_{\p_3})} 
& ({\rm if}~\ell=3)
  \end{array} \right. \\
& = \zeta_{\ell}^{{\tilde{\chi}_{2}
^{z}(\tilde{\sigma}_{\p_3})}
+{\tilde{\chi}_{2}^{1-z}
(\tilde{\sigma}_{\p_3})}} \\
& = \zeta_{\ell}^{{\tilde{\chi}_{2}
^{z, \gamma}(\tilde{\sigma}_{\p_3})}
+{\tilde{\chi}_{2}^{1-z, \gamma'}
(\tilde{\sigma}_{\p_3})}}.
\end{align*}
By the functional equation 
in Theorem~\ref{func},
the above last side is equal to
\[
\zeta_{\ell}
^{-\rho_{z, \gamma}(\tilde{\sigma}_{\p_3})
\rho_{1-z, \gamma'}(\tilde{\sigma}_{\p_3})
+\frac{1}{24}
(\chi(\tilde{\sigma}_{\p_3})^2-1)}
=1~\hspace{3mm}
~({\rm by~~(\ref{Kl}),
(\ref{sigma})}).
\]
This completes the proof.
\end{proof}
\vspace{3mm}


\appendix
\section{\bf \boldmath Examples - Case of $\ell=3$ with the assumption~(A)}
In this appendix,
we present examples of $[\p_1,\p_2,\p_3]_{3}$ and $\ell i_{2}^{(3)}
(z)(\tilde{\sigma}_{\p_3}) 
~{\rm mod}~3$ for some pairs
$(p_1,p_2)$ which have
$\alpha_{p_1,p_2}$ 
satisfying the assumption (A) in Section~\ref{3}.
The rational primes $p$ which satisfy $p \equiv 1~{\rm mod}~9$ and $1 \leq -p \leq 1000$
are the following $29$ numbers:
\begin{equation}
    \textbf{L}:=\left\{\parbox{27em}{$-17,
-53,
-71,
-89,
-107,
-179,
-197,
-233,
-251,
-269,$\\
$-359,
-431,
-449,
-467,
-503,
-521,
-557,
-593,
-647,
-683,$\\
$-701,
-719,
-773,
-809,
-827,
-863,
-881,
-953,
-971$}\right\}.
\end{equation}

For any pair $(p_1,p_2)$ of distinct rational primes in $\textbf{L}$,
prime ideals $\p_1=(p_1),
\p_2=(p_2)$ of $\bQ(\zeta_3)$ satisfy the conditions (\ref{pc3}),
(\ref{assume3})
and (\ref{pcc3}).

In  [AMM;~Example 6.4],
F. Amano
showed that 
one can take $\alpha_{-17,-53}=8-3\sqrt[3]{17}$ satisfying (A) in the case where 
$(p_1,
p_2)=(-17,
-53)$,
and gave values 
of $[\p_1,\p_2,\p_3]_3$
for $p_3=-71,
-89,
-107,
-179,
-197$.
According to [HM;~Example 4.2.15],
Y. Mizusawa found other pairs 
$(p_1,p_2)$
$=(-17,-467),$
$(-107,-449),$
$(-431,-233)$
which have $\alpha_{p_1,p_2} \in O_{K_1}$ with the assumption (A).
For example,
one can take
\[\alpha_{-17,-469}=6-9\sqrt[3]{-17},\]
\[\alpha_{-107,-449}=-24-5\sqrt[3]{-107},\]
\[\alpha_{-431,-233}=-68-9\sqrt[3]{-431}\]
respectively for these cases.

Let us now give new examples.
Consider the case where $(p_1,p_2)=(-17,
-593)$.
In this case,
we can take
\begin{equation}\label{sol}
x=9,
y=2,
w=-1
\end{equation}
as a solution of (\ref{xyw3})
and
\[\alpha_{p_1,p_2}\left(=\alpha_{-17,-593}\right)=9+2\sqrt[3]{-17}\] satisfying (A).
Hence
\[\theta^{(3)}_{p_1,p_2}\left(=\theta^{(3)}_{-17,-593}\right)=(9+2\zeta_3\sqrt[3]{-17})
(9+2\zeta_3^2\sqrt[3]{-17})^2.\]
Moreover,
let $\p_3=(p_3)$ be a prime ideal of $\bQ(\zeta_3)$
which satisfies (\ref{p_33}).
Then,
\begin{align*}
  [\p_1,\p_2,\p_3]_{3}
  & = \dfrac{{\sigma}_{{\p_3}}(\sqrt[3]{\theta^{(3)}_{p_1,p_2}})}
{\sqrt[3]{\theta^{(3)}_{p_1,p_2}}}~~~~({\rm by}~~{\rm Theorem}~\ref{222}) \\
  & \equiv {\theta^{(3)}_{p_1,p_2}}^{\frac{p_3^2-1}{3}}~{\rm mod}~\tilde{\p}_3,
\end{align*}
where $\tilde{\p}_3$ is a prime ideal 
of $K_{p_1,p_2}^{(3)}$ above $\p_3$.
Since ${\theta^{(3)}_{p_1,p_2}}^{\frac{p_3^2-1}{3}} \in K_1=\bQ(\zeta_3, \sqrt[3]{p_1})$,
\[[\p_1,\p_2,\p_3]_{3} \equiv {\theta^{(3)}_{p_1,p_2}}^{\frac{p_3^2-1}{3}}~{\rm mod}~\tilde{\p}_3 \cap K_1.\]
Therefore,
we obtain the following test:
for $c=0,
1,
-1,$
\begin{equation}\label{test}
[\p_1,\p_2,\p_3]_3=\zeta_3^{c}~~ \Longleftrightarrow~~N_{K_1/\bQ}({\theta^{(3)}_{p_1,p_2}}^{\frac{p_3^2-1}{3}}-\zeta_3^{c}) \equiv 0~{\rm mod}~p_3.
\end{equation}

On the other hand,
we can take 
$\alpha_{p_2,p_1}\left(=\alpha_{-593,-17}\right)=9+\sqrt[3]{-593}$ satisfying (A).
Hence
$\theta^{(3)}_{p_2,p_1}\left(=\theta^{(3)}_{-593,-17}\right)=(9+\zeta_3\sqrt[3]{-593})
(9+\zeta_3^2\sqrt[3]{-593})^2$.
By replacing $\theta^{(3)}_{p_1,p_2}$ (resp. $K_1$) with $\theta^{(3)}_{p_2,p_1}$ (resp. $K_2=\bQ(\zeta_3, \sqrt[3]{p_2})$)  in (\ref{test}),
we obtain the following test:
for $c=0,
1,
-1$,
\begin{equation}\label{test2}
[\p_2,\p_1,\p_3]_3=\zeta_3^{c}~~ \Longleftrightarrow~~N_{K_2/\bQ}({\theta^{(3)}_{p_2,p_1}}^{\frac{p_3^2-1}{3}}-\zeta_3^{c}) \equiv 0~{\rm mod}~p_3.
\end{equation}
Checking the right hand condition of (\ref{test}) and (\ref{test2}) by PARI/GP,
we can compute $[\p_1,\p_2,\p_3]_3$ and $[\p_2,\p_1,\p_3]_3$.
Furthermore,
by combining with Theorem~\ref{tl} and Theorem~\ref{aaa},
we can also compute $\ell i_{2}^{(3)}
(z)(\tilde{\sigma}_{\p_3}) 
~{\rm mod}~3$
and
$\ell i_{2}^{(3)}
(1-z)(\tilde{\sigma}_{\p_3}) 
~{\rm mod}~3$
where
\[z\left(=-p_1\frac{y^3}{x^3}\right)=\dfrac{136}{729}.\]
Consequently,
for $p_3 \in {\bf L} \backslash \{-17, -593\}$,
we get TABLE~\ref{table1}.
Thus, we can be assured that the reciprocity law $[\p_1,\p_2,\p_3]_3 
\cdot [\p_2,\p_1,\p_3]_3=1$ of Theorem \ref{recip3} holds.

Let us examine more general behaviors of
\[[\p_{\rho(1)},\p_{\rho(2)},\p_{\rho(3)}]_3\] where $\rho \in S_3$ is any permutation of the set $\{1,
2,
3\}$,
in the cases
\[\{p_1,p_2,p_3\}=\{-17,-53,-431\},~\{-17,-557,-773\},~\{-17,-593,-773\}.\]
Finding a solution of (\ref{xyw3}) and checking the test (\ref{test}) for each case,
we can compute $[\p_{\rho(1)},\p_{\rho(2)},\p_{\rho(3)}]_3$.
Furthermore,
by combining with Corollary~\ref{c3},
we can also compute $\ell i_{2}^{(3)}
(z)(\tilde{\sigma}_{\p_3}) 
~{\rm mod}~3$.
Consequently,
we get TABLE~\ref{table2}.
Based on TABLE~\ref{table2},
it may be plausible to expect that
\begin{equation}\label{conjecture}
[\p_{\rho(1)},\p_{\rho(2)},\p_{\rho(3)}]_3=[\p_1, \p_2, \p_3]_3^{{\rm sgn}(\rho)},
\end{equation}
where ${\rm sgn}(\rho) \in \{1,-1\}$ is the signature of $\rho \in S_3$,
for $(\p_1,
\p_2,
\p_3)$ satisfying the conditions (\ref{pc3}),
(\ref{assume3}),
(\ref{pcc3}) and (\ref{p_33}).

\newpage

\vspace*{\stretch{1}}
\begin{table}[h]
\caption{Table of $[\p_1, \p_2, \p_3]_3$,
$[\p_2, \p_1, \p_3]_3$,
$\ell i_{2}^{(3)}
(z)(\tilde{\sigma}_{\p_3}) 
~{\rm mod}~3$
and
$\ell i_{2}^{(3)}
(1-z)(\tilde{\sigma}_{\p_3}) 
~{\rm mod}~3$
for $(p_1, p_2)=(-17,
-593),
p_3 \in {\bf L} \backslash \{p_1, p_2\}$}
\label{table1}
\renewcommand{\arraystretch}{1.2}
\begin{tabular}{|c||c|c|c|c|} 
\hline
$p_3$ & $[\p_1, \p_2, \p_3]_3$ & $[\p_2, \p_1, \p_3]_3$ & $\ell i_{2}^{(3)}
(z)(\tilde{\sigma}_{\p_3}) 
~{\rm mod}~3$ & $\ell i_{2}^{(3)}
(1-z)(\tilde{\sigma}_{\p_3}) 
~{\rm mod}~3$  \\ \hline \hline
$-53$ & $\zeta_3$ & $\zeta_3^{-1}$ & $-1$ & 1 \\ \hline
$-71$ & $\zeta_3^{-1}$ & $\zeta_3$ & 1 & $-1$ \\ \hline
$-89$ & $\zeta_3$ & $\zeta_3^{-1}$ & $-1$ & 1  \\ \hline
$-107$ & $\zeta_3^{-1}$ & $\zeta_3$ & 1 & $-1$ \\ \hline
$-179$ & $\zeta_3$ & $\zeta_3^{-1}$ & $-1$ & 1 \\ \hline
$-197$ & $\zeta_3$ & $\zeta_3^{-1}$ & $-1$ & 1 \\ \hline
$-233$ & $1$ & $1$ & 0 & 0 \\ \hline
$-251$ & $\zeta_3$ & $\zeta_3^{-1}$ & $-1$ & 1 \\ \hline
$-269$ & $\zeta_3$ & $\zeta_3^{-1}$ & $-1$ & 1 \\ \hline
$-359$ & $\zeta_3^{-1}$ & $\zeta_3$  & 1 & $-1$ \\ \hline
$-431$ & $\zeta_3^{-1}$ & $\zeta_3$  & 1 & $-1$ \\ \hline
$-449$ & $1$ & $1$ & 0 & 0 \\ \hline
$-467$ & $\zeta_3$ & $\zeta_3^{-1}$ & $-1$ & 1 \\ \hline
$-503$ & $\zeta_3^{-1}$ & $\zeta_3$ & 1 & $-1$ \\ \hline
$-521$ & $\zeta_3$ & $\zeta_3^{-1}$ & $-1$ & 1 \\ \hline
$-557$ & $\zeta_3$ & $\zeta_3^{-1}$ & $-1$ & 1 \\ \hline
$-647$ & $\zeta_3^{-1}$ & $\zeta_3$  & 1 & $-1$ \\ \hline
$-683$ & $\zeta_3^{-1}$ & $\zeta_3$ & 1 & $-1$ \\ \hline
$-701$ & $\zeta_3^{-1}$ & $\zeta_3$ & 1 & $-1$ \\ \hline
$-719$ & $\zeta_3^{-1}$ & $\zeta_3$ & 1 & $-1$ \\ \hline
$-773$ & $\zeta_3^{-1}$ & $\zeta_3$ & 1 & $-1$ \\ \hline
$-809$ & $\zeta_3^{-1}$ & $\zeta_3$ & 1 & $-1$ \\ \hline
$-827$ & $1$ & $1$ & 0 & 0 \\ \hline
$-863$ & $\zeta_3$ & $\zeta_3^{-1}$ & $-1$ & 1 \\ \hline
$-881$ & $\zeta_3^{-1}$ & $\zeta_3$ & 1 & $-1$ \\ \hline
$-953$ & $\zeta_3$ & $\zeta_3^{-1}$ & $-1$ & 1 \\ \hline
$-971$ & $1$ & $1$ & 0 & 0 \\ \hline
\end{tabular}
\end{table}
\vspace*{\stretch{2}}
\pagebreak

\newpage

\vspace*{\stretch{1}}
\begin{table}[h]\label{table2}
\caption{Table of $[\p_{1},\p_{2},\p_{3}]_3$ and $\ell i_{2}^{(3)}
(z)(\tilde{\sigma}_{\p_3}) 
~{\rm mod}~3$
for the cases of $\{p_1,p_2,p_3\}=\{-17,-53,-431\},~\{-17,-557,-773\},~\{-17,-593,-773\}$}
\label{table2}
  \begin{center}
\renewcommand{\arraystretch}{1.4}
\begin{tabular}{|c|c|c|c|c|c|} \hline
$(p_1,p_2)$ & $\alpha_{p_1,p_2}=x+y\sqrt[3]{p_1}$ & $z=-p_1\frac{y^3}{x^3}$ & $p_3$ & $[\p_1,\p_2,\p_3]_3$ & $\ell i_{2}^{(3)}
(z)(\tilde{\sigma}_{\p_3})
~{\rm mod}~3$ \\ \hline \hline
$(-17,-53)$ & $8+3\sqrt[3]{-17}$ & $\frac{459}{512}$ & $-431$ & $1$ & $0$ \\ \hline
$(-53,-17)$ & $8+\sqrt[3]{-53}$ &$\frac{53}{512}$ & $-431$ & $1$ & 0 \\ \hline
$(-17,-431)$ & $31+15\sqrt[3]{-17}$ &$\frac{57375}{29791}$ & $-51$ & $1$ & 0 \\ \hline
$(-431,-53)$ & $10+3\sqrt[3]{-431}$ & $\frac{11637}{1000}$ & $-17$ & $1$ & 0 \\ \hline
$(-53,-431)$ & $10-\sqrt[3]{-53}$ & $-\frac{53}{1000}$ & $-17$ & $1$ & 0 \\ \hline
$(-431,-17)$ & $31-4\sqrt[3]{-431}$ & $-\frac{27584}{29791}$ & $-53$ & $1$ & 0 \\ \hline \hline
$(-17,-557)$ & $-42-16\sqrt[3]{-17}$ &$\frac{8704}{9261}$ & $-773$ & $\zeta_3$ & $-1$ \\ \hline
$(-557,-17)$ & $-42-2\sqrt[3]{-557}$ & $\frac{557}{9261}$ & $-773$ & $\zeta_3^{-1}$ & $1$ \\ \hline
$(-17,-773)$ & $-23+8\sqrt[3]{-17}$ &$-\frac{8704}{12167}$ & $-557$ & $\zeta_3^{-1}$ & 1 \\ \hline
$(-773,-557)$ & $-6-\sqrt[3]{-773}$ & $\frac{773}{216}$ & $-17$ & $\zeta_3^{-1}$ & 1 \\ \hline
$(-557,-773)$ & $-6+\sqrt[3]{-557}$ & $-\frac{557}{216}$ & $-17$ & $\zeta_3$ & $-1$ \\ \hline
$(-773,-17)$ & $-23-3\sqrt[3]{-773}$ & $\frac{20871}{12167}$ & $-557$ & $\zeta_3$ & $-1$ \\ \hline \hline
$(-17,-593)$ & $9+2\sqrt[3]{-17}$ & $\frac{136}{729}$ & $-773$ & $\zeta_3^{-1}$ & 1 \\ \hline
$(-593,-17)$ & $9+\sqrt[3]{-593}$ & $\frac{593}{729}$ & $-773$ & $\zeta_3$ & $-1$ \\ \hline
$(-17,-773)$ & $-23+8\sqrt[3]{-17}$ & $-\frac{8704}{12167}$ & $-593$ & $\zeta_3$ & $-1$ \\ \hline
$(-773,-593)$ & $-55-6\sqrt[3]{-773}$ & $\frac{166968}{166375}$ & $-17$ & $\zeta_3$ & $-1$ \\ \hline
$(-593,-773)$ & $-55+\sqrt[3]{-593}$ & $-\frac{593}{166375}$ & $-17$ & $\zeta_3^{-1}$ & 1 \\ \hline
$(-773,-17)$ & $-23-3\sqrt[3]{-773}$ & $\frac{20871}{12167}$ & $-593$ & $\zeta_3^{-1}$ & 1 \\ \hline
  \end{tabular}
\end{center}
\end{table}
\vspace*{\stretch{2}}
\pagebreak

\newpage



\end{document}